\documentclass{article}

\usepackage[a4paper]{geometry}

\title{Gauss-Seidel Method with Oblique Direction}
\author{Fang Wang, Weiguo Li$^*$, Wendi Bao, Zhonglu Lv\\
\small China University of Petroleum\\[-0.8ex]
\small Qingdao, China\\
\small\tt $^*$liwg@upc.edu.cn\\
}

\date{June 2021}

\usepackage{graphicx}
\usepackage{charter}
\usepackage[charter]{mathdesign}
\usepackage{algorithm}
\usepackage{amsthm}
\newtheorem{theorem}{Theorem}
\newtheorem{lemma}{Lemma}
\newtheorem{example}{Example}
\usepackage{algorithmicx}
\usepackage{algpseudocode}
\usepackage{booktabs}
\usepackage{graphicx}
\usepackage{subfigure}
\usepackage{amsmath}

\usepackage{multirow}
\usepackage{algorithm,algpseudocode,float}
\usepackage{lipsum}

\makeatletter
\newenvironment{breakablealgorithm}
  {
   \begin{center}
     \refstepcounter{algorithm}
     \hrule height.8pt depth0pt \kern2pt
     \renewcommand{\caption}[2][\relax]{
       {\raggedright\textbf{\ALG@name~\thealgorithm} ##2\par}%
       \ifx\relax##1\relax 
         \addcontentsline{loa}{algorithm}{\protect\numberline{\thealgorithm}##2}%
       \else 
         \addcontentsline{loa}{algorithm}{\protect\numberline{\thealgorithm}##1}%
       \fi
       \kern2pt\hrule\kern2pt
     }
  }{
     \kern2pt\hrule\relax
   \end{center}
  }
\makeatother
\begin{document}
\bibliographystyle{plain}
\maketitle
\begin{abstract}
In this paper, a Gauss-Seidel method with oblique direction (GSO) is proposed for finding the least-squares solution to a system of linear equations, where the coefficient matrix may be full rank or rank deficient and the system is overdetermined or underdetermined. Through this method, the number of iteration steps and running time can be reduced to a greater extent to find the least-squares solution,
especially when the columns of matrix A are close to linear correlation. It is theoretically proved that GSO method converges to the least-squares solution.
At the same time, a randomized version--randomized Gauss-Seidel method with oblique direction (RGSO) is established, and its convergence is proved. Theoretical proof and numerical results show that the GSO method and the RGSO method are more efficient than the coordinate descent (CD) method and the randomized coordinate descent (RCD) method.\\

{\bf Key words:} linear least-squares problem, oblique direction, coordinate descent method, randomization, convergence property.
\end{abstract}
\section{Introduction}
\label{intro}
\hspace{1.5em}Consider a linear least-squares problem

\begin{align}\label{e1}
\mathop{arg\min}\limits_{x \in R^n}\|Ax-b\|^2,
\end{align}
where $b\in R^m$ is a real $m$ dimensional vector, and the columns of coefficient matrix $A\in R^{m\times n}$ are non-zero, which doesn't lose the generality of matrix $A$. Here and in the sequel, $\|\cdot\|$ indicates the Euclidean norm of a vector. When $A$ is full column rank, $(\ref{e1})$ has a unique solution $x^*=A^{\dag}b=(A^TA)^{-1}A^Tb$, where $A^{\dag}$ and $A^T$ are the Moore-Penrose pseudoinverse \cite{Be74} and the transpose of $A$, respectively.
One of the iteration methods that can be used to solve $(\ref{e1})$ economically and effectively is the coordinate descent (CD) method \cite{D10}, which is also obtained by applying the classical Gauss-Seidel iteration method to the following normal equation (see \cite{Ru83})
\begin{align}\label{e2}
A^TAx=A^Tb.
\end{align}

In solving $(\ref{e1})$, the CD method has a long history of development, and is widely used in various fields, such as machine learning \cite{CH08}, biological feature selection \cite{BH11}, tomography \cite{BS96,YW99}, and so on. Inspired by the randomized coordinate descent (RCD) method proposed by Leventhal and Lewis \cite{D10} and its linear convergence rate analyzed theoretically \cite{MN15}, a lot of related work such as the randomized block versions \cite{LX15,NN17,RT14} and greedy randomized versions of the CD method \cite{BW19,ZG20} have been developed and studied. For more information about a variety of randomized versions of the coordinate descent method, see \cite{NS17,ST11,W15} and the references therein. These methods mentioned above are based on the CD method, and can be extended to Kaczmarz-type methods. The recent work of Kaczmarz-type method can be referred to \cite{GL20,DH19,XY21,ZJ19,BW199,YC19,CH21}. Inspired by the above work, We propose a new descent direction based on the construction idea of the CD method, which is formed by the weighting of two coordinate vectors. Based on this, we propose a Gauss-Seidel method with oblique direction (GSO) and construct the randomized version--randomized Gauss-Seidel method with oblique direction (RGSO), and analyze the convergence properties of the two methods.

Regarding our proposed methods--the GSO method and the RGSO method, we emphasize the efficiency when the columns of matrix $A$ are close to linear correlation. In \cite{NW13}, it is mentioned that when the rows of matrix $A$ are close to linear correlation, the convergence speed of the K method and the randomized Kaczmarz method \cite{SV09} decrease significantly. Inspired by the above phenomena, we experimented the convergence performance of the CD method and the RCD method when the columns of matrix A are close to linear correlation and it is found through experiments that the theoretical convergence speed and experimental convergence speed of the CD method and the the RCD method will be greatly reduced. The exponential convergence in expectation of the RCD method is as follows:
\begin{align}\label{e3}
E_k\delta(x^{(k+1)})\leq \left(1-\frac{1}{\kappa_F^2(A)}\right)\delta(x^{(k)}),
\end{align}
where $\delta(x)=F(x)-minF$, $F(x)=\|Ax-b\|^2$. Here and in the sequel, $||A||_2=\mathop{max}\limits_{\|x\|=1}||Ax||$, $||A||_F$, $\kappa_F(A)=||A||_F\cdot||A^ \dag||_2$ are used to denote Euclidean norm, Frobenius norm and the scaled condition number of the matrix $A$, respectively. The subgraph (a) in Figure \ref{f1} shows that when the column of matrix $A$ is closer to the linear correlation, $\kappa_F^2(A)$ will  become larger, which further reduce the convergence rate of the RCD method. The subgraph (b) in Figure \ref{f1} illustrates the sensitivity of the CD method and the RCD method to linear correlation column of $A$. This further illustrates the necessity of solving this type of problem, and the GSO method and the RGSO method we proposed can be used effectively to solve that one. For the initial data setting, explanation of the experiment in Figure \ref{f1} and the experiment on this type of matrix, please refer to Section 4 in this paper.
\begin{figure}\label{f1}
  \centering
  \subfigure[]{
    \includegraphics[width=2.15in]{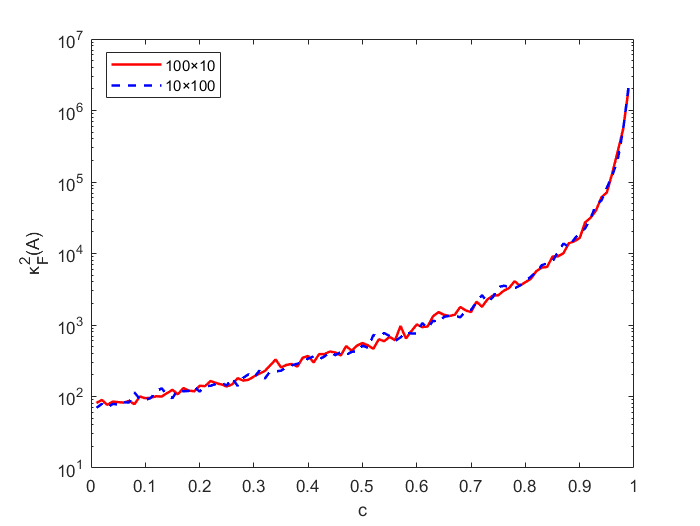}
  }
  \subfigure[]{
    \includegraphics[width=2.15in]{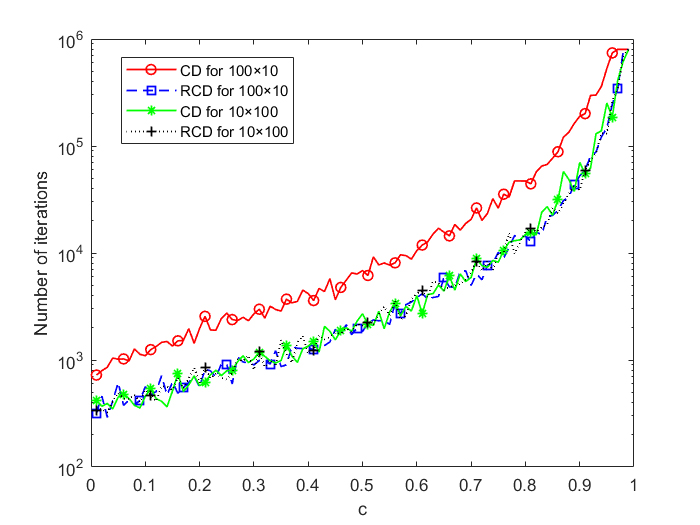}
  }
  \caption{Matrix $A$ is generated by the $rand$ function in the interval $[c,1]$.
   (a): $\kappa_F^2(A)$ of matrix $A$ changes with $c$.  (b): When the system is consistent, the number of iterations for the CD method and the RCD method to converge with the change of c, where the maximum number of iterations is limited to 800,000.}
\end{figure}

In this paper, $\langle \cdot \rangle$ stands for the scalar product, and we indicate by $e_i$ the column vector with 1 at the ith position and 0 elsewhere. In addition, for a given matrix $G=(g_{ij})\in R^{m\times n}$, $g_i^T$, $G_j$ and $\sigma_{min}(G)$, are used to denote its ith row, jth column and the smallest nonzero singular value of $G$ respectively. Given a symmetric positive semidefinite matrix $G\in R^{n\times n}$, for any vector $x\in R^n$ we define the corresponding seminorm as $||x||_G=\sqrt{x^TGx}$. Let $E_k$ denote the expected value conditonal on the first k iterations, that is,
\begin{equation*}
E_k[\cdot]=E[\cdot|j_0,j_1,...,j_{k-1}],
\end{equation*}
where $j_s(s=0,1,...,k-1)$ is the column chosen at the sth iteration.

The organization of this paper is as follows. In Section 2, we introduce the CD method and its construction idea. In Section 3, we propose the GSO method naturally and get its randomized version--RGSO method, and prove the convergence of the two methods. In Section 4, some numerical examples are provided to illustrate the effectiveness of our new methods. Finally, some brief concluding remarks are described in Section 5.

\section{Coordinate Descent Method}\label{sec:2}
\hspace{1.5em}Consider a linear system
\begin{align}\label{e4}
\tilde{A}x=b,
\end{align}
where the coefficient matrix $\tilde{A}\in R^{n\times n}$ is a positive-definite matrix, and $b\in R^n$ is a real m dimensional vector. $\tilde{x}^*=\tilde{A}^{-1}b$ is the unique solution of $(\ref{e4})$. In this case, solving $(\ref{e4})$ is equivalent to the following strict convex quadratic minimization problem
\begin{align*}
f(x)=\frac{1}{2}x^T\tilde{A}x-b^Tx.
\end{align*}
From \cite{D10}, the next iteration point $x^{(k+1)}$ is the solution to $\mathop{min}\limits_{\substack{t \in R}}f(x^{(k)}+td)$, i.e.
\begin{align}\label{e5}
x^{(k+1)}=x^{(k)}+\frac{(b-\tilde{A}x^{(k)})^{T}d}{d^T\tilde{A}d}d,
\end{align}
where $d$ is a nonzero direction, and $x^{(k)}$ is a current iteration point. It is easily proved that
\begin{align}\label{e6}
f(x^{(k+1)})-f(\tilde{x}^*)=\frac{1}{2} \|x^{(k+1)}-\tilde{x}^* \|^2_{\tilde{A}}=\frac{1}{2}\|x^{(k)}-\tilde{x}^*\|^2_{\tilde{A}}-\frac{((b-\tilde{A}x^{(k)})^Td)^2}{2d^T\tilde{A}d}.
\end{align}
One natural choice of a set of easily computable search directions is to choose $d$ by successively cycling through the set of canonical unit vectors $\left\{e_1,...,e_n\right\}$, where $e_i \in R^n,i=1,\cdots,n$. When $A\in R^{m\times n}$ is full column rank, we can apply $(\ref{e2})$ to $(\ref{e5})$ to get:
$$x^{(k+1)}=x^{(k)}+\frac{\langle r^{(k)},A_i\rangle}{||A_i||^2}e_i,$$
where $i=mod(k,n)+1$. This is the iterative formula of CD method, also known as Gauss-Seidel method. This method is linearly convergent but with a rate not easily expressible in terms of typical matrix quantities. See \cite{B96,GL96,QS07}. The CD method can only ensure that one entry of $A^Tr$ is 0 in each iteration, i.e. $A_i^Tr^{(k)}=0\ \ (i=mod(k,n)+1)$. In the next chapter, we propose a new oblique direction $d$ for $(\ref{e5})$, which is the weight of the two coordinate vectors, and use the same idea to get a new method-- the GSO method. The GSO method can ensure that the two entries of $A^Tr$ are 0 in each iteration, thereby accelerating accelerating the convergence.

{\bf Remark 1.} When $\tilde{A}$ is positive semidefinite matrix, $(\ref{e4})$ may not have a unique solution, replace $\tilde{x}^*$ with any least-squares solution, $(\ref{e5})$, $(\ref{e6})$ still hold, if $d^T\tilde{A}d\neq 0$.

{\bf Remark 2.} \ The Kaczmarz method can be regarded as a special case of $(\ref{e5})$ under a different regularizing linear system
\begin{align}\label{e7}
AA^Ty=b, \ \ x=A^Ty,
\end{align}
when $d$ is  selected cyclically  through the set of canonical unit vectors $\left\{e_1,...,e_m\right\}$, where $e_{i}\in R^m$, $i=1,2,\cdots,m$.

\section{Gauss-Seidel Method with Oblique Direction and its Randomized Version }
\subsection{Gauss-Seidel Method with Oblique Direction}
\hspace{1.5em}We propose a similar $d$, that is $d=e_{i_{k+1}}-\frac{\langle A_{i_{k+1}},A_{i_{k}}\rangle}{||A_{i_{k}}||^2}e_{i_{k}}$, where $e_{i}\in R^n$, $i=1,2,\cdots,n$. Using $(\ref{e2})$ and $(\ref{e5})$, we get
\begin{align}\label{e8}
\nonumber x^{(k+1)}&=x^{(k)}+
\frac{(A^Tb-A^TAx^{(k)})^T(e_{i_{k+1}}-\frac{\langle A_{i_{k+1}},A_{i_{k}}\rangle}{||A_{i_{k}}||^2}e_{i_{k}})}
{\left\|A(e_{i_{k+1}}-\frac{\langle A_{i_{k+1}},A_{i_{k}}\rangle}{||A_{i_{k}}||^2}e_{i_{k}})\right\|^2}(e_{i_{k+1}}-\frac{\langle A_{i_{k+1}},A_{i_{k}}\rangle}{||A_{i_{k}}||^2}e_{i_{k}})\\
&=x^{(k)}+\frac{A_{i_{k+1}}^Tr^{(k)}-\frac{\langle A_{i_{k+1}},A_{i_{k}}\rangle}{||A_{i_{k}}||^2}A_{i_{k}}^Tr^{(k)}}{||A_{i_{k+1}}||^2-\frac{\langle A_{i_{k+1}},A_{i_{k}}\rangle^2}{||A_{i_{k}}||^2}}(e_{i_{k+1}}-\frac{\langle A_{i_{k+1}},A_{i_{k}}\rangle}{||A_{i_{k}}||^2}e_{i_{k}}).
\end{align}
Now we prove that $A_{i_{k}}^Tr^{(k)}=0.$
\begin{align*}
A_{i_{k}}^Tr^{(k)}&=\langle A_{i_{k}},b-Ax^{(k)}\rangle\\
&=\langle A_{i_{k}},r^{(k-1)}\rangle-\langle A_{i_{k}},r^{(k-1)}\rangle+\frac{\langle A_{i_{k}},A_{i_{k-1}}\rangle}{||A_{i_{k-1}}||^2} A_{i_{k-1}}^Tr^{(k-1)}\\
&=\frac{\langle A_{i_{k}},A_{i_{k-1}}\rangle}{||A_{i_{k-1}}||^2}A_{i_{k-1}}^Tr^{(k-1)}, \ \ k=2,3,...
\end{align*}

We only need to guarantee $A_{i_{1}}^Tr^{(1)}=0$, so we need to take the simplest coordinate descent projection as the first step. $(\ref{e8})$ becomes
\begin{align}
\nonumber x^{(k+1)}=x^{(k)}+\frac{A_{i_{k+1}}^Tr^{(k)}}{||A_{i_{k+1}}||^2-\frac{\langle A_{i_{k+1}},A_{i_{k}}\rangle^2}{||A_{i_{k}}||^2}}(e_{i_{k+1}}-\frac{\langle A_{i_{k+1}},A_{i_{k}}\rangle}{||A_{i_{k}}||^2}e_{i_{k}}).
\end{align}

The algorithm is described in Algorithm $\ref{A3.1}$. Without losing generality, we assume that all columns of $A$ are not zero vectors.

\begin{breakablealgorithm}\label{A3.1}
\caption{Gauss-Seidel method with Oblique Projection (GSO)}
  \label{alg:Framwork2}
  \begin{algorithmic}[1]
    \Require
      $A\in R^{m\times n}$, $b\in R^{m}$, $x^{(0)}\in R^n$, $K$, $\varepsilon>0$
    \State For $i=1:n$, $N(i)=\|A_i\|^2$
    \State Compute $r^{(0)}=b-Ax^{(0)}$, $\alpha_0=\frac{\langle A_1,r^{(0)}\rangle}{N(1)}$, $x^{(1)}=x^{(0)}+\alpha_{0}e_1$, $r^{(1)}=r^{(0)}-\alpha_{0}A_1$, and set $i_{k+1}=1$
    \For {$k=1,2,\cdots, K-1$}
    \State Set $i_k=i_{k+1}$ and choose a new $i_{k+1}$: $i_{k+1}=mod(k,n)+1$
    \State Compute $G_{i_k}=\langle A_{i_{k}}, A_{i_{k+1}}\rangle$ \ and $g_{i_k}=N(i_{k+1})-\frac{G_{i_k}}{N(i_k)}G_{i_k}$
    \If {$g_{i_k}>\varepsilon$}
    \State Compute $\alpha_{k}=\frac{\langle A_{i_{k+1}},r^{(k)}\rangle}{g_{i_{k}}}$ \ and $\beta_k=-\frac{G(i_k)}{N(i_k)}\alpha_k$
    \State Compute $x^{(k+1)}=x^{(k)}+\alpha_ke_{i_{k+1}}+\beta_ke_{i_{k}}$, and $r^{(k+1)}=r^{(k)}-\alpha_kA_{i_{k+1}}-\beta_kA_{i_k}$
    \EndIf
    \EndFor
    \State Output $x^{(K)}$
  \end{algorithmic}
\end{breakablealgorithm}

It's easy to get
\begin{align*}
A_{i_{k-1}}^Tr^{(k)}&=A_{i_{k-1}}^T(r^{(k-1)}-\alpha_{k-1}A_{i_{k}}-\beta_{k-1}A_{i_{k-1}})\\
&=A_{i_{k-1}}^T(r^{(k-1)}-\frac{\langle A_{i_{k}},r^{(k-1)}\rangle}{g_{i_{k-1}}}A_{i_{k}}+\frac{\langle A_{i_{k-1}},A_{i_{k}}\rangle \langle A_{i_k},r^{(k-1)}\rangle}{||A_{i_{k-1}}||^2g_{i_{k}}}A_{i_{k-1}})\\
&=0.\ \ k=2,3,\cdots
\end{align*}
The last equality holds due to $A_{i_{k}}^Tr^{(k)}=0, k=1,2,\cdots$.
Before giving the proof of the convergence of the GSO method, we redefine the iteration point. For $x^{(0)}\in R^n$ as initial approximation, we define $x^{(0,0)}$, $x^{(0,1)},...,x^{(0,n)}\in R^n$ by
\begin{equation}\label{e3.2}
\left\{
             \begin{array}{lr}
             x^{(0,0)}=x^{(0)}+\frac{A_1^T(b-Ax^{(0)})}{||A_1||^2}e_1,&\\
             x^{(0,1)}=x^{(0,0)}+\frac{A_{2}^T(b-Ax^{(0,0)})}{||A_2||^2-\frac{\langle A_2,A_1\rangle^2}{||A_1||^2}}(e_2-\frac{\langle  A_{2},A_{1}\rangle}{||A_1||^2}e_1),&\\
             x^{(0,2)}=x^{(0,1)}+\frac{A_{3}^T(b-Ax^{(0,1)})}{||A_3||^2-\frac{\langle A_3,A_2\rangle^2}{||A_2||^2}}(e_3-\frac{\langle  A_{3},A_{2}\rangle}{||A_2||^2}e_2),&\\
             \cdots \cdots \cdots \cdots \cdots \cdots \cdots \cdots \cdots &\\
             x^{(0,n-1)}=x^{(0,n-2)}+\frac{A_{n}^T(b-Ax^{(0,n-2)})}{||A_n||^2-\frac{\langle A_n,A_{n-1}\rangle^2}{||A_{n-1}||^2}}(e_n-\frac{\langle  A_{n},A_{n-1}\rangle}{||A_{n-1}||^2}e_{n-1}),&\\
             x^{(0,n)}=x^{(0,n-1)}+\frac{A_{1}^T(b-Ax^{(0,n-1)})}{||A_1||^2-\frac{\langle A_1,A_{n}\rangle^2}{||A_{n}||^2}}(e_1-\frac{\langle  A_{1},A_{n}\rangle}{||A_{n}||^2}e_{n}).&\\
             \end{array}
\right.
\end{equation}
For convenience, denote $A_{n+1}=A_1,$ $b_{n+1}=b_1$. When the iteration point $x^{(p,n)}$ ${\forall p\geq0}$ is given, the iteration points are obtained continuously by the following formula
\begin{equation}\label{e3.3}
\left\{
             \begin{array}{lr}
             for \;\; i=1:n &  \\
             x^{(p+1,i)}=x^{(p+1,i-1)}+\frac{A_{i+1}^T(b-Ax^{(p+1,i-1)})}{||A_{i+1}||^2-\frac{\langle A_{i+1},A_{i}\rangle^2}{||A_i||^2}}(e_{i+1}-\frac{\langle  A_{i+1},A_{i}\rangle}{||A_{i}||^2}e_{i}),\\
             end
             \end{array}
\right.
\end{equation}
where $x^{(p+1,0)}=x^{(p,n)}$. Then, we can easily obtain that $x^{(k+1)}=x^{(p,i)}$, and $A_{i_k}^Tr^{(k)}=A_{i+1}^Tr^{(p,i)}=0$, if $k=p\cdot n+i,0\leq i<n.$

The convergence of the GSO is provided as follows.

\begin{theorem}\label{t3.1}
Consider $(\ref{e1})$, where the coefficient $A\in R^{m\times n}$, $b \in R^m$ is a given vector, and $\tilde{x}$ is any least-squares solution of $(\ref{e1})$. Let $x^{(0)}\in R^n$ be an arbitrary initial approximation, then the sequence $\{Ax^{(k)}\}_{k=1}^{\infty}$  generated by the GSO
algorithm is convergent , and satisfy the following equation:
\begin{equation}\label{e3.4}
\lim_{k\rightarrow\infty}\|x^{(k)}-\tilde{x}\|_{A^TA}= 0.
\end{equation}
\end{theorem}
\begin{proof}
According to $(\ref{e3.2})$-$(\ref{e3.3})$ we obtain the sequence of approximations (from top to bottom and left to right, and by also using the notational convention $x^{(p+1,0)}= x^{(p,n)}$).
\begin{equation*}
\left\{
             \begin{array}{lr}
             x^{(0)},x^{(0,0)}&  \\
             x^{(0,1)},x^{(0,2)},\cdots,x^{(0,n)}=x^{(1,0)} &  \\
             x^{(1,1)},x^{(1,2)},\cdots,x^{(1,n)}=x^{(2,0)} &  \\
             \cdots \cdots \cdots \cdots \cdots \cdots \cdots \cdots \cdots&\\
             x^{(p,1)},x^{(p,2)},\cdots,x^{(p,n)}=x^{(p+1,0)} &  \\
             \cdots \cdots \cdots \cdots \cdots \cdots \cdots \cdots \cdots
             \end{array}
\right.
\end{equation*}
Apply $(\ref{e2})$ to $(\ref{e6})$,we get
\begin{align}\label{e3.5}
||x^{(p,i+1)}-\tilde{x}||^2_{A^TA}=||x^{(p,i)}-\tilde{x}||^2_{A^TA}-\frac{((A^Tb-A^TAx^{(p,i)})^Td)^2}{d^TA^TAd},
\end{align}
where $d=e_{i+2}-\frac{\langle A_{i+2},A_{i+1}\rangle}{||A_{i+1}||^2}e_{i+1}$, $e_{i}\in R^n,i=1,\cdots,n.$
Obviously, the sequence $\{\|x^{(p,i)}-x^*\|_{A^TA}\}_{p=0, i=0}^{\infty, n-1}$, i.e. $\{\|x^{(k+1)}-x^*\|_{A^TA}\}_{k=1}^{\infty}$ is a monotonically decreasing sequence with lower bounds. There exists a $\alpha \ge 0 $ such that
\begin{equation}\label{e3.6}
  \lim_{p\rightarrow\infty}\|x^{(p,i)}-\tilde{x}\|_{A^TA}=\alpha\ge 0, \ \ \forall \ i=0,1,\cdots,n-1.
\end{equation}
Thus, take the limit of $p$ on both sides of $(\ref{e3.5})$, and because $i$ was arbitrary we apply $A_{i_{k+1}}^Tr^{(p,i)}=0$, and get
\begin{equation}\label{e3.7}
\lim_{p\rightarrow\infty}A_{i+2}^T r^{(p,i)}= 0,\ \ \forall \ i=0,1,\cdots,n-1.
\end{equation}
The residuals satisfy
$$r^{(p,i)}=r^{(p,i)}-\frac{\langle A_{i+2},r^{(p,i)}\rangle}{||A_{i+2}||^2-\frac{\langle A_{i+2},A_{i+1}\rangle^2}{||A_{i+1}||^2}}\left(A_{i+2}-\frac{\langle A_{i+2},A_{i+1}\rangle}{||A_{i+1}||^2}A_{i+1}\right).$$
Taking the limit of p on both sides of the above equation, we get
$$\lim_{p\rightarrow\infty}r^{(p,i+1)}=\lim_{p\rightarrow\infty}r^{(p,i)},\ \ \forall \ i=0,1,\cdots,n-1.$$
Using the above equation and $(\ref{e3.7})$, we can easily deduce that
\begin{equation}\label{e3.8}
  \lim_{p\rightarrow\infty}A^T r^{(p,i)}= 0,\ \ \forall \ i=0,1,\cdots,n-1.
\end{equation}
Because the sequence $\{\|x^{(p,i)}-\tilde{x}\|_{A^TA}\}_{p=0, i=0}^{\infty, n-1}$ is bounded, we obtain
\begin{equation}\label{e3.9}
  \|x^{(p,i)}\|_{A^TA}\le \|\tilde{x}\|_{A^TA}+\|x^{(p,i)}-\tilde{x}\|_{A^TA}\le \|\tilde{x}\|_{A^TA}+\|x^{(0,1)}-\tilde{x}\|_{A^TA},\,\, \forall p\ge 0.
\end{equation}
According to $(\ref{e3.9})$ we get that the sequence $\{Ax^{(p,0)}\}_{p=0}^{\infty}$ is bounded, thus there exists a convergent subsequence $\{Ax^{(p_j, 0)}\}_{j=1}^{\infty}$, let's denote it as
\begin{equation}\label{e3.10}
  \lim_{j\rightarrow\infty}Ax^{(p_j, 0)}=\hat{b}.
\end{equation}
From $(\ref{e3.2})$-$(\ref{e3.3})$, we get
\begin{equation*}
x^{(p_j,1)}=x^{(p_j,0)}-\frac{A^T_2(b-Ax^{(p_j,0)})}{||A_2||^2-\frac{\langle A_2,A_1\rangle^2}{||A_1||^2}}\left(e_2-\frac{\langle A_2,A_1\rangle}{||A_1||^2}e_1\right),\,\,\,\forall \ j>0.
\end{equation*}
By multiplying the both sides of the above equation left by matrix $A$ and using $(\ref{e3.7})$, we can get that
\begin{equation*}
  \lim_{j\rightarrow\infty}Ax^{(p_j, 1)}=\hat{b}.
\end{equation*}
With the same way we obtain
\begin{equation*}
  \lim_{j\rightarrow\infty}Ax^{(p_j,i)}=\hat{b},\,\,\,\,\forall \ i=0, 1,\cdots, n-1.
\end{equation*}
Then, from $(\ref{e3.8})$ we get for any $i=1, \cdots, n-1$,
\begin{equation*}
  \lim_{j\rightarrow\infty} A^Tr^{(p_j,i)}=A^T(b-\hat{b})=0.
\end{equation*}
From $(\ref{e3.6})$ and the above equation, we get
\begin{align*}
\lim_{j\rightarrow\infty}\|x^{(p_j,i)}-\tilde{x}\|_{A^TA}=\alpha=0, \ \  \ \forall \ i=0,1,\cdots,n-1
\end{align*}
Hence,
$$\lim_{p\rightarrow\infty}\|x^{(p,i)}-\tilde{x}\|_{A^TA}=0,  \ \  \ \forall \ i=0,1,\cdots,n-1$$
then $(\ref{e3.4})$ holds.
\end{proof}

{\bf Remark 3.} \ When $g_{i_{k}}=0$, $A_{i_{k+1}}$ is parallel to $A_{i_k}$, i.e. $\exists \lambda >0$, s.t. $A_{i_k}=\lambda A_{i_{k+1}}$. According to the above derivation, the GSO method is used to solve (1.2) which is consistent, so the following equation holds:
$$A_{i_k}^Tb=\lambda A_{i_{k+1}}^Tb,$$
 which means for $(\ref{e2})$ the $i_k$th equation: $\langle A_{i_{k}},Ax\rangle=A^T_{i_{k}}b$, and the $i_{k+1}$th equation: $\langle A_{i_{k+1}},Ax\rangle=A^T_{i_{k+1}}b$ are coincident,
and we can skip this step without affecting the final calculation to obtain the least-squares solution. When $g_{i_{k}}$ is too small, it is easy to produce large errors in the process of numerical operation, and we can regard it as the same situation as $g_{i_{k}}=0$ and skip this step.

{\bf Remark 4.} \ By the GSO method, we have: $||x^{(k+1)}-\tilde{x}||^2_{A^TA}=||x^{(k)}-\tilde{x}||^2_{A^TA}-\frac{( A_{i_{k+1}}^Tr^{k})^2}{g_{i_k}}$, where $g_{i_k}=\|A_{i_{k+1}}\|^2
-\frac{\langle A_{i_{k+1}},A_{i_k}\rangle^2}{\|A_{i_k}\|^2}$. But the CD method holds: $||x^{(k+1)}-\tilde{x}||^2_{A^TA}=||x^{(k)}-\tilde{x}||^2_{A^TA}-\frac{(A_{i_{k+1}}^Tr^{k})^2}{||A_{i_{k+1}}||^2}$. So the GSO method is faster than the CD method unless $\langle A_{i_k},A_{i_{k+1}}\rangle =0$. When $\langle A_{i_k},A_{i_{k+1}}\rangle =0$, the convergence rate of the GSO method is the same as that of the CD method. This means that when the coefficient matrix $A$ is a column orthogonal matrix, the GSO method degenerates to the CD method.

{\bf Remark 5.} \  The GSO method needs $8m+5$ floating-point operations per step, and the CD method needs $4m+1$ floating-point operations per step.

{\bf Remark 6.} \ When the matrix $A$ is full column rank, let $x^*$ be the unique least-squares solution of $(\ref{e1})$, the sequence $\{x^{(k)}\}_{k=1}^{\infty}$  generated by the GSO
method holds: $\lim\limits_{k\rightarrow\infty}\|x^{(k)}-x^*\|_{A^TA}=0$, that is, $\lim\limits_{k\rightarrow\infty}\|A(x^{(k)}-x^*)\|^2=0$. Therefore,
$$\lim_{k\rightarrow\infty}\|x^{(k)}-x^*\|^2=0.$$

\begin{example}\label{EX3.1}\upshape
Consider the following systems of linear equations
\begin{equation}\label{e3.11}
\left\{
\begin{array}{rrr}
5x_1+45x_2&=&50,\\
9x_1+80x_2&=&89,
\end{array}
\right.
\end{equation}
\begin{equation}\label{e3.12}
\left\{
\begin{array}{rrr}
x_1+11x_2&=& 12,\\
-2x_1-21x_2&=&-23,\\
3x_1+32x_2&=&35
\end{array}
\right.
\end{equation}
and
\begin{equation}\label{e3.13}
\left\{
\begin{array}{rrr}
x_1+9x_2&=&0,\\
4x_1+36x_2&=&42.5,\\
13x_1+118x_2&=&131,
\end{array}
\right.
\end{equation}
\end{example}
$(\ref{e3.11})$ is square and consistent, $(\ref{e3.12})$ is overdetermined and consistent, and $(\ref{e3.13})$ is overdetermined and inconsistent. Vector $x^*=(1,1)^T$ is the unique solution to the above $(\ref{e3.11})$ and $(\ref{e3.12})$, is the unique least-squares solution to $(\ref{e3.13})$. It can be found that the column vectors of these systems are close to linearly correlated. Numerical experiments show that they take $650259$, $137317$, $3053153$ steps respectively for the CD method to be applied to the above systems to reach the relative solution error requirement $\frac{\|x^{(k)}-x^*\|^2}{\|x^*\|^2}\leq\frac{1}{2}\times 10^{-6}$, but the GSO method can find the objective solutions to the above three systems in one step.
\subsection{Randomized Gauss-Seidel  Method with Oblique Direction}
\hspace{1.5em}If the columns whose residual entries are not 0 in algorithm \ref{A3.1} are selected uniformly and randomly, we get a randomized Gauss-seidel method with oblique direction (RGSO) and its convergence as follows.
 \begin{breakablealgorithm}\label{A3.2}
\caption{Randomized Gauss-Seidel Method with Oblique Direction (RGSO)}
  \label{alg:Framwork4}
  \begin{algorithmic}[1]
    \Require
      $A\in R^{m\times n}$, $b\in R^{m}$, $x^{(0)}\in R^n$, $K$, $\varepsilon>0$
    \State For $i=1:n$, $N(i)=\|A_i\|^2$
    \State Randomly select $i_1$, and compute $r^{(0)}=b-Ax^{(0)}$, $\alpha_0=\frac{\langle A_{i_{1}},r^{(0)}\rangle}{||A_{i_{1}}||^2}$, and $x^{(1)}=x^{(0)}+\alpha_0 e_{i_{1}}$
    \State Randomly select $i_2\neq i_1$, and compute $r^{(1)}=r^{(0)}-\alpha_0A_{i_{1}}$, $\alpha_1=\frac{\langle A_{i_{2}},r^{(1)}\rangle}{||A_{i_{2}}||^2-\frac{\langle A_{i_{1}},A_{i_{2}}\rangle^2}{||A_{i_{1}}||^2}}$, and $\beta_1=-\frac{\langle A_{i_{1}},A_{i_{2}}\rangle}{||A_{i_{1}}||^2}\alpha_1$
    \State Compute $x^{(2)}=x^{(1)}+\alpha_1e_{i_{2}}+\beta_1e_{i_{1}}$,and $r^{(2)}=r^{(1)}-\alpha_1A_{i_2}-\beta_1A_{i_1}$
    \For {$k=2,3,\cdots, K-1$}
    \State Randomly select $i_{k+1}$ $(i_{k+1}\neq i_k,i_{k-1})$
    \State Compute $G_{i_k}=\langle A_{i_{k}}, A_{i_{k+1}}\rangle$ \ and $g_{i_k}=N(i_{k+1})-\frac{G_{i_k}}{N(i_k)}G_{i_k}$
    \If {$g_{i_k}>\varepsilon$}
    \State Compute $\alpha_{k}=\frac{\langle A_{i_{k+1}},r^{(k)}\rangle}{g_{i_{k}}}$ \ and $\beta_k=-\frac{G(i_k)}{N(i_k)}\alpha_k$
    \State Compute $x^{(k+1)}=x^{(k)}+\alpha_ke_{i_{k+1}}+\beta_ke_{i_{k}}$, and $r^{(k+1)}=r^{(k)}-\alpha_kA_{i_{k+1}}-\beta_kA_{i_k}$
    \EndIf
    \EndFor
    \State Output $x^{(K)}$
  \end{algorithmic}
\end{breakablealgorithm}

\begin{lemma}\label{l3.1}
Consider $(\ref{e1})$, where the coefficient $A\in R^{m\times n}$ , $b \in R^m$ is a given vector, and $\tilde{x}$ is any  solution to $(\ref{e1})$ , then we obtain the bound on the following expected conditional on the first $k$ $(k\geq 2)$ iteration of the RGSO
   $$E_k\frac{(A_{i_{k+1}}^Tr^{(k)})^2}{g_{i_{k}}}\geq \frac{1}{n-2}\frac{\sigma_{min}^2(A)||\tilde{x}-x^{(k)}||^2_{A^TA}}{||A||^2_F-\sigma_{min}^2(A)}.$$
\end{lemma}
\begin{proof}
For the RGSO mthod, it is easy to get that $A_{i_{k}}^Tr^{(k)}=0\ \ (k=1,2,\cdots )$ and $A_{i_{k-1}}^Tr^{(k)}=0\ \ (k=2,3,\cdots)$ are still valid.
\begin{align*}
E_k\frac{(A_{i_{k+1}}^Tr^{(k)})^2}{g_{i_{k}}}&=\frac{1}{n-2}\sum\limits_{\substack{s=1\\s\not=i_k,i_{k-1}}}^n \frac{(A_s^Tr^{(k)})^2}{||A_s||^2-\frac{\langle A_s,A_{i_{k}}\rangle^2}{||A_{i_k}||^2}}\\
&\geq\frac{1}{n-2}\frac{\sum\limits_{\substack{s=1,s\not=i_k,i_{k-1}}}^n(A_s^Tr^{(k)})^2}{\sum\limits_{\substack{s=1\\s\not=i_k,i_{k-1}}}^n(||A_s||^2-\frac{\langle A_s,A_{i_{k}}\rangle^2}{||A_{i_{k}}||^2})}\left(=\frac{1}{n-2}\frac{\sum\limits_{\substack{s=1}}^n(A_s^Tr^{(k)})^2}{\sum\limits_{\substack{s=1}}^n(||A_s||^2-\frac{\langle A_s,A_{i_{k}}\rangle^2}{||A_{i_{k}}||^2})}\right)\\
&=\frac{1}{n-2}\frac{||A^TA(\tilde{x}-x^{(k)})||^2}{||A||^2_F-\frac{||A^TA_{i_{k}}||^2}{||A_{i_{k}}||^2}}\\
&\geq \frac{1}{n-2}\frac{\sigma_{min}^2(A)||\tilde{x}-x^{(k)}||^2_{A^TA}}{||A||^2_F-\sigma_{min}^2(A)},\ \ k=2,3,...
\end{align*}
The first inequality uses the conclusion of $\frac{|b_1|}{|a_1|}+\frac{|b_2|}{|a_2|}\geq\frac{|b_1|+|b_2|}{|a_1|+|a_2|}$ (if $|a_1|>0$, $|a_2|>0$), and the second one uses the conclusion of $\|A^Tz\|_2^2\geq\sigma_{min}^2(A)\|z\|_2^2$, if $z\in R(A)$.
\end{proof}
\begin{theorem}\label{t3.2}
Consider $(\ref{e1})$, where the coefficient $A\in R^{m\times n}$ , $b \in R^m$ is a given vector, and $\tilde{x}$ is any least-squares solution of $(\ref{e1})$. Let $x^{(0)}\in R^n$ be an arbitrary initial approximation,
and define the least-squares residual and error by
  $$F(x)=||Ax-b||^2,$$
  $$\delta(x)=F(x)-minF,$$
then the RGSO method is linearly convergent in expectation to a solution in $(\ref{e1})$.

For each iteration:$k=2,3,...$,
\begin{align*}
E_k\delta(x^{(k+1)})\leq\left(1-\frac{1}{(n-2)(k_F^2(A)-1)} \right) \delta(x^{(k)}).
\end{align*}
In particular, if $A$ has full column rank, we have the equivalent property
\begin{align*}
E_k\left[\|x^{(k+1)}-x^*\|^2_{A^TA}\right] \leq \left(1-\frac{1}{(n-2)(k_F^2(A)-1)} \right)\|x^{(k)}-x^*\|^2_{A^TA},
\end{align*}
where $x^*=A^{\dag} b=(A^TA)^{-1}A^Tb$ is the unique least-squares solution.
\end{theorem}
\begin{proof}
It is easy to prove that
$$F(x)-F(\tilde{x})=||x-\tilde{x}||^2_{A^TA}=\delta(x).$$
Apply $(\ref{e2})$ to $(\ref{e6})$ with $d=e_{i_{k+1}}-\frac{\langle A_{i_{k+1}},A_{i_{k}}\rangle}{||A_{i_k}||^2}e_{i_k}$ and $A_{i_{k}}^Tr_k=0,k=1,2,..$ ,we get that
\begin{align*}
F(x^{(k+1)})-F(\tilde{x})&=||x^{(k+1)}-\tilde{x}||^2_{A^TA}\\
&=||x^{(k)}-\tilde{x}||^2_{A^TA}-\frac{(A^T_{i_{k+1}}r^{(k)})^2}{g_{i_{k}}}.
\end{align*}
Making conditional expectation on both sides, and applying Lemma $\ref{l3.1}$, we get
\begin{align*}
E_k\left[F(x^{(k+1)})-F(\tilde{x})\right]&=||x^{(k)}-\tilde{x}||^2_{A^TA}-E_k\left[\frac{(A_{i_{k+1}}^Tr^{(k)})^2}{g_{i_{k}}}\right]\\
&\leq ||x^{(k)}-\tilde{x}||^2_{A^TA}-\frac{\sigma_{min}^2(A)||\tilde{x}-x^{(k)}||^2_{A^TA}}{(n-2)(||A||^2_F-\sigma_{min}^2(A))},
\end{align*}
that is
\begin{align*}
E_k\delta(x^{(k+1)})&\leq \left(1-\frac{\sigma_{min}^2(A)}{(n-2)(||A||_F^2-\sigma_{min}^2(A))}\right)\delta(x^{(k)})\\
&=\left(1-\frac{1}{(n-2)(k_F^2(A)-1)} \right) \delta(x^{(k)}).
\end{align*}
If A has full column rank, the solution in $(\ref{e1})$ is unique and the $\tilde{x}=x^*$. Thus, we get
\begin{align*}
E_k\left[\|x^{(k+1)}-x^*\|^2_{A^TA}\right] \leq \left(1-\frac{1}{(n-2)(k_F^2(A)-1)} \right)\|x^{(k)}-x^*\|^2_{A^TA}.
\end{align*}
\end{proof}

{\bf Remark 7.} \ In particular, after unitizing the columns of matrix A, we can get from Lemma $\ref{l3.1}$:
\begin{align*}
E_k\frac{(A_{i_{k+1}}^Tr^{(k)})^2}{g_{i_{k}}}&=\frac{1}{n-2}\sum\limits_{\substack{s=1\\s\not=i_k,i_{k-1}}}^n \frac{(A_s^Tr^{(k)})^2}{||A_s||^2-\frac{\langle A_s,A_{i_{k}}\rangle^2}{||A_{i_k}||^2}}\\
&=\sum\limits_{\substack{s=1\\s\not=i_k,i_{k-1}}}^n \frac{\|A_s\|^2}{\|A\|^2_F-2}\frac{(A_s^Tr^{(k)})^2}{||A_s||^2-\frac{\langle A_s,A_{i_{k}}\rangle^2}{||A_{i_k}||^2}}\\
&\geq\sum\limits_{\substack{s=1\\s\not=i_k,i_{k-1}}}^n \frac{1}{\|A\|^2_F-2}\frac{(A_s^Tr^{(k)})^2}{1-\gamma_{i_{k}}^2}\\
&\geq\frac{\sigma_{min}^2(A)||\tilde{x}-x^{(k)}||^2_{A^TA}}{(1-\gamma_{i_{k}}^2)(\|A\|^2_F-2)},k=2,3,...
\end{align*}
where $ \gamma_{i_{k}}=\mathop{min}\limits_{s\neq i_{k},i_{k-1}}|\langle A_s, A_{i_{k}}\rangle|$. Then we get from Theorem \ref{t3.1}:
\begin{align*}
E_k\delta(x^{(k+1)})&\leq \left(1-\frac{\sigma_{min}^2(A)}{(1-\gamma^2_{i_{k}})(||A||_F^2-2)}\right)\delta(x^{(k)}).
\end{align*}
Comparing the above equation with $(\ref{e3})$, we can get that under the condition of column unitization, the RGSO method is theoretically faster than the RCD method. Note that by Remark 3, we can avoid the occurrence of $ \gamma_{i_{k}}=1$.

\section{Numerical Experiments}
\hspace{1.5em}In this section, some numerical examples are provided to illustrate the effectiveness of the coordinate descent (CD) method, the Gauss-Seidel method with oblique direction (GSO), the randomized coordinate descent (RCD) method (with uniform probability) and the randomized Gauss-Seidel method with oblique direction (RGSO) for solving $(\ref{e1})$. All experiments are carried out using MATLAB (version R2019b) on a personal computer with 1.60 GHz central processing unit (Intel(R) Core(TM) i5-10210U CPU), 8.00 GB memory, and Windows operating system (64 bit Windows 10).

Obtained from \cite{PP12}, the least-squares solution set for $(\ref{e1})$ is
\begin{align*}
LSS(A;b)=S(A,b_A)=\{P_{N(A)}(x^{(0)})+x_{LS},x^{(0)}\in R^n\},
\end{align*}
where $LSS(A;b)$ is the set of all least solutions to $(\ref{e1})$, and $x_{LS}$ is the unique least-squares solution of minimal Euclidean norm. For the consistent case $b\in R(A)$, $LSS(A;b)$ will be denoted by $S(A;b)$. If $b=b_A+b_A^*$, with
$$b_A=P_{R(A)}(b),b_A^*=P_{N(A^T)}(b),$$
where $P_s$ denotes the orthogonal projection onto the vector subspace $S$ of some $R^q$. From Theorem $\ref{t3.1}$ and Theorem $\ref{t3.2}$, we can know that the sequence$\|x^{(k)}-\tilde{x}\|^2_{A^TA}$ generated by the GSO method and the RGSO method converges to $0$. Due to
\begin{align*}
\|x^{(k)}-\tilde{x}\|^2_{A^TA}&=\|A(x^{(k)}-\tilde{x})\|^2\\
&=\|b_A+b_A^*-r^{(k)}-b_A\|^2\\
&=\|b_A^*-r^{(k)}\|^2,
\end{align*}
where $b_A^*$ can be known in the experimental hypothesis, and $r^{(k)}$ is calculated in the iterative process, we can propose a iteration termination rule: The methods are terminated once residual relative error  ($RRE$), defined by $$RRE=\frac{\|b_A^*-r^{(k)}\|^2}{\|b\|^2}$$ at the current iterate $x^{(k)}$, satisfies $RRE<\frac{1}{2}\times10^{-6}$ or the maximum iteration steps $500,000$ being reached. If the number of iteration steps exceeds $500,000$, it is denoted as "-". IT and CPU are the medians of the required iterations steps and the elapsed CPU times with respect to $50$ times repeated runs of the corresponding method. To give an intuitive demonstration of the advantage, we define the speed-up as follows:
 $$\operatorname{speed-up}^{1}=\frac{\operatorname{CPU\ of\ CD}}{\operatorname{CPU\ of\ GSO}}, \operatorname{speed-up}^{2}=\frac{\operatorname{CPU\ of\ RGS}}{\operatorname{CPU\ of\ RGSO}}.$$

In our implementations, all iterations are started from the initial guess $x_0=zeros(n,1)$.
First, set a least-squares solution $\tilde{x}$, which is generated by using the MATLAB function rand.
Then set $b_A=A\tilde{x}$. When linear system is consistent, $b_A^*=0$, $b=b_A$, else $b_A^*\in null(A^T)$, $b=b_A+b_A^*$. When the column of the coefficient matrix A is full rank, the methods can converge to the only least-squares solution $x^*$ under the premise of convergence.

\subsection{Experiments for Random Matrix Collection in $[0,1]$}
\hspace{1.5em}The random matrix collection in $[0,1]$ is randomly generated by using the MATLAB function $rand$, and the numerical results are reported in Tables 1-9. According to the characteristics
of the matrix generated by  MATLAB function $rand$, Table \ref{table1} to Table \ref{table3}, Table \ref{table4} to Table \ref{table6}, Table \ref{table7} to Table \ref{table9} are  the
experiments respectively for the overdetermined consistent linear systems, overdetermined inconsistent linear systems, and underdetermined consistent linear systems. In Table \ref{table1}
 to Table \ref{table6}, under the premise of convergence, all methods can find the unique least-squares solution $x^*$, i.e. $x^*=(A^TA)^{-1}A^Tb$. In Table \ref{table7} to Table \ref{table9},
 all methods can find the least-squares solution under the premise of convergence, but they can't be sure to find the same least-squares solution.

From these tables, we see that the GSO method and the RGSO method are more outstanding than the CD method and the RCD method respectively in terms of both IT and CPU  with significant speed-up, regardless of whether the corresponding linear system is consistent or inconsistent. We can observe that in Tables 1-6, for the overdetermined linear systems, whether it is consistent or inconsistent, CPU and IT of all methods increase with the increase of $n$, and the CD method is extremely sensitive to the increase of $n$. When n increases to 100, it stops because it exceeds the maximum number of iterations. In Tables 7-9, for the underdetermined  consistent linear system, CPU and IT of all methods increase with the increase of $m$.
\begin{table}[!htbp]
\centering
\caption{IT and CPU of CD, RCD, GSO and RGSO for $m \times n$ matrices $A$ with $n = 50$ and different $m$ when the overdetermined linear system is consistent}
\label{table1}
\renewcommand\arraystretch{0.75}
\begin{tabular}{cccccccc}
\toprule
\multicolumn{2}{c}{ $ m \times n $ }& $1000 \times 50$ & $2000 \times 50 $ & $3000 \times 50$ & $4000 \times 50$ & $5000 \times 50$ \\
\hline
\multirow{2}*{CD}   & IT          &73004	&74672	&74335	&74608	&74520\\
                     & CPU       &0.1605 	&0.3082 	&0.5200 	&0.9833 	&1.3256 \\
\hline
\multirow{2}*{GSO}  & IT          &11110	&11081	&10915	&10951	&10934\\
                     & CPU        &0.0379 	&0.0711 	&0.1224 	&0.2412 	&0.3244\\
\hline
\multicolumn{2}{c}{ speed-up$^1$ }   &4.23 	&4.33 	&4.25 	&4.08 	&4.09 \\
\hline
\multirow{2}*{RCD}  & IT          &1733	&1596	&1505	&1583	&1522\\
                     & CPU         &0.0125 	&0.0151 	&0.0196 	&0.0322 	&0.0416\\
\hline\multirow{2}*{RGSO}  & IT      &778	&752	&789	&700	&685 \\
                     & CPU    &0.0070 	&0.0086 	&0.0145 	&0.0210 	&0.0267 \\
\hline
\multicolumn{2}{c}{ speed-up$^2$ }    &1.78 	&1.75 	&1.36 	&1.53 	&1.56\\
 \bottomrule

\end{tabular}
\end{table}

 \begin{table}[!htbp]
\centering
\caption{IT and CPU of CD, RCD, GSO and RGSO for $m \times n$ matrices $A$ with $n = 100$ and different $m$ when the overdetermined linear system is consistent}
\label{table2}
\renewcommand\arraystretch{0.75}
\begin{tabular}{cccccccc}
\toprule
\multicolumn{2}{c}{ $ m \times n $ }& $1000 \times 100$ & $2000 \times 100 $ & $3000 \times 100$ & $4000 \times 100$ & $5000 \times 100$ \\
\hline
\multirow{2}*{CD}   & IT             &   -	 &     -	 &     -	 &     -	&      -\\
                     & CPU          &   -	 &     -	 &     -	 &     -	&      - \\
\hline
\multirow{2}*{GSO}  & IT          &84180	&81595	&80120	&80630	&79131\\
                     & CPU         &0.2945 	&0.5315 	&0.9227 	&1.7860 	&2.6375\\
\hline
\multicolumn{2}{c}{ speed-up$^1$ }       &   -	 &     -	 &     -	 &     -	&      - \\
\hline
\multirow{2}*{RCD}  & IT          &3909	&3304	&3564	&3391	&3187\\
                     & CPU         &0.0278 	&0.0318 	&0.0475 	&0.0719 	&0.0957\\
\hline\multirow{2}*{RGSO}  & IT        &1657	&1598	&1486	&1751	&1432\\
                     & CPU        &0.0148 	&0.0204 	&0.0264 	&0.0546 	&0.0631 \\
\hline
\multicolumn{2}{c}{ speed-up$^2$ }    &1.88 	&1.56 	&1.80 	&1.32 	&1.52\\
 \bottomrule
\end{tabular}
\end{table}

 \begin{table}[!htbp]
\centering

\caption{IT and CPU of CD, RCD, GSO and RGSO for $m \times n$ matrices $A$ with $n = 150$ and different $m$ when the overdetermined linear system is consistent}
\label{table3}
\renewcommand\arraystretch{0.75}
\begin{tabular}{cccccccc}
\toprule
\multicolumn{2}{c}{ $ m \times n $ }& $1000 \times 150$ & $2000 \times 150 $ & $3000 \times 150$ & $4000 \times 150$ & $5000 \times 150$ \\
\hline
\multirow{2}*{CD}   & IT          &    -	&      -	&      -	&      -	&      -\\
                     & CPU       &    -	&      -	&      -	&      -	&      - \\
\hline
\multirow{2}*{GSO}  & IT           &276537	&270070	&260799	&259227	&259033\\
                     & CPU        &0.9292 	&1.6746 	&2.7657 	&5.9676 	&9.1506\\
\hline
\multicolumn{2}{c}{ speed-up$^1$ }    &    -	&      -	&      -	&      -	&      - \\
\hline
\multirow{2}*{RCD}  & IT         &6781	&5375	&5371	&5288	&5358\\
                     & CPU       &0.0472 	&0.0486 	&0.0660 	&0.1095 	&0.1712\\
\hline\multirow{2}*{RGSO}  & IT       &2880	&2574	&2466	&2352	&2547 \\
                     & CPU        &0.0241 	&0.0304 	&0.0415 	&0.0741 	&0.1195\\
\hline
\multicolumn{2}{c}{ speed-up$^2$ }    &1.96 	&1.60 	&1.59 	&1.48 	&1.43\\
 \bottomrule
\end{tabular}
\end{table}

 \begin{table}[!htbp]
\centering
\caption{IT and CPU of CD, RCD, GSO and RGSO for $m \times n$ matrices $A$ with $n = 50$ and different $m$ when the overdetermined linear system is inconsistent}
\label{table4}
\renewcommand\arraystretch{0.75}
\begin{tabular}{cccccccc}
\toprule
\multicolumn{2}{c}{ $ m \times n $ }& $1000 \times 50$ & $2000 \times 50 $ & $3000 \times 50$ & $4000 \times 50$ & $5000 \times 50$ \\
\hline
\multirow{2}*{CD}   & IT          &73331	&73895	&73910	&74810	&74606\\
                     & CPU        &0.1591 	&0.3004 	&0.5266 	&1.0081 	&1.4170  \\
\hline
\multirow{2}*{GSO}  & IT          &11124	&10955	&10875	&10984	&10910\\
                     & CPU        &0.0442 	&0.0716 	&0.1337 	&0.2411 	&0.3376\\
\hline
\multicolumn{2}{c}{ speed-up$^1$ }    &3.60 	&4.20 	&3.94 	&4.18 	&4.20 \\
\hline
\multirow{2}*{RCD}  & IT          &1736	&1786	&1706	&1599	&1514\\
                     & CPU         &0.0129 	&0.0164 	&0.0244 	&0.0338 	&0.0414\\
\hline\multirow{2}*{RGSO}  & IT      &744	&718	&762	&737	&769 \\
                     & CPU        &0.0067 	&0.0087 	&0.0142 	&0.0223 	&0.0327 \\
\hline
\multicolumn{2}{c}{ speed-up$^2$ }     &1.91 	&1.88 	&1.72 	&1.52 	&1.26\\
 \bottomrule
\end{tabular}
\end{table}

\begin{table}[!htbp]
\centering
\caption{IT and CPU of CD, RCD, GSO and RGSO for $m \times n$ matrices $A$ with $n = 100$ and different $m$ when the overdetermined linear system is inconsistent}
\label{table5}
\renewcommand\arraystretch{0.75}
\begin{tabular}{cccccccc}
\toprule
\multicolumn{2}{c}{ $ m \times n $ }& $1000 \times 100$ & $2000 \times 100 $ & $3000 \times 100$ & $4000 \times 100$ & $5000 \times 100$ \\
\hline
\multirow{2}*{CD}   & IT          & -&	       -	&       -	  &     -	&       -\\
                     & CPU        &  -	&       -	&       -	 &      -	&       -  \\
\hline
\multirow{2}*{GSO}  & IT         &84415	&84104	&80361	&79462	&79572\\
                     & CPU        &0.2829 	&0.5457 	&0.9160 	&1.7187 	&2.5587\\
\hline
\multicolumn{2}{c}{ speed-up$^1$ }    &  -	&       -	&       -	 &      -	&       - \\
\hline
\multirow{2}*{RCD}  & IT         &3973	&3511	&3599	&3092	&3221\\
                    & CPU        &0.0305 	&0.0329 	&0.0473 	&0.0615 	&0.0943\\
\hline\multirow{2}*{RGSO}  & IT     &1676	&1675	&1596	&1456	&1563\\
                     & CPU       &0.0142 	&0.0203 	&0.0279 	&0.0427 	&0.0666\\
\hline
\multicolumn{2}{c}{ speed-up$^2$ }     &2.14 	&1.62 	&1.70 	&1.44 	&1.42\\
 \bottomrule
 \end{tabular}
\end{table}

\begin{table}[!htbp]
\centering
\caption{IT and CPU of CD, RCD, GSO and RGSO for $m \times n$ matrices $A$ with $n = 150$ and different $m$ when the overdetermined linear system is inconsistent}
\label{table6}
\renewcommand\arraystretch{0.75}
\begin{tabular}{cccccccc}
\toprule
\multicolumn{2}{c}{ $ m \times n $ }& $1000 \times 150$ & $2000 \times 150 $ & $3000 \times 150$ & $4000 \times 150$ & $5000 \times 150$ \\
\hline
\multirow{2}*{CD}   & IT                & -	    &   -	  &     -	   &    -	   &    -\\
                     & CPU              & -	    &   -	  &     -	   &    -	   &    -  \\
\hline
\multirow{2}*{GSO}  & IT          &288578	&267841	&265105	&262289	&258320\\
                     & CPU        &1.0080 	&1.7435 	&2.9230 	&5.8877 	&7.8848\\
\hline
\multicolumn{2}{c}{ speed-up$^1$ }       &    -	   &    -	  &     -	&       -	   &    -\\
\hline
\multirow{2}*{RCD}  & IT        &6799	&5690	&5340	&4860	&4979\\
                     & CPU       &0.0478 	&0.0520 	&0.0690 	&0.0977 	&0.1390\\
\hline\multirow{2}*{RGSO}  & IT      &2834	&2472	&2463	&2475	&2368\\
                    & CPU      &0.0247 	&0.0300 	&0.0467 	&0.0739 	&0.0979 \\
\hline
\multicolumn{2}{c}{ speed-up$^2$ }      &1.94 	&1.73 	&1.48 	&1.32 	&1.42\\
\bottomrule
\end{tabular}
\end{table}

\begin{table}[!htbp]
\centering
\caption{IT and CPU of CD, RCD, GSO and RGSO for $m \times n$ matrices $A$ with $n = 1000$ and different $m$ when the underdetermined linear system is consistent}
\label{table7}
\renewcommand\arraystretch{0.75}
\begin{tabular}{cccccccc}
\toprule
\multicolumn{2}{c}{ $ m \times n $ }& $100 \times 1000$ & $200 \times 1000 $ & $300 \times 1000$ & $400 \times 1000$ & $500 \times 1000$ \\
\hline
\multirow{2}*{CD}   & IT               &3805	&11193&	22638	&43868	&82643\\
                     & CPU              &0.0025 &	0.0089 	&0.0215 &	0.0499& 	0.1102 \\
\hline
\multirow{2}*{GSO}  & IT          &1621	&3544	&6824	&12339	&24149\\
                     & CPU        &0.0016 &	0.0044 &	0.0111 &	0.0224 	&0.0507\\
\hline
\multicolumn{2}{c}{ speed-up$^1$ }       &1.52 &	2.01 	&1.93 	&2.23 	&2.17\\
\hline
\multirow{2}*{RCD}  & IT          &4113&	10926&	21267	&39220&	70545\\
                     & CPU      &0.0210 &	0.0593 &	0.1151 	&0.2219 &	0.4207\\
\hline\multirow{2}*{RGSO}  & IT       &1876	&4152	&7985	&13158&	24441\\
                    & CPU      &0.0102 &	0.0253 	&0.0497 &	0.0877 &	0.1680 \\
\hline
\multicolumn{2}{c}{ speed-up$^2$ }     &2.05 &	2.34 &	2.31 &	2.53 &	2.50\\
\bottomrule
\end{tabular}
\end{table}

\begin{table}[!htbp]
\centering
\caption{IT and CPU of CD, RCD, GSO and RGSO for $m \times n$ matrices $A$ with $n = 2000$ and different $m$ when the underdetermined linear system is consistent}
\label{table8}
\renewcommand\arraystretch{0.75}
\begin{tabular}{cccccccc}
\toprule
\multicolumn{2}{c}{ $ m \times n $ }& $100 \times 2000$ & $200 \times 2000 $ & $300 \times 2000$ & $400 \times 2000$ & $500 \times 2000$ \\
\hline
\multirow{2}*{CD}   & IT              &3285	&7790	&13913	&21575	&32445\\
                     & CPU             &0.0029 	&0.0071 	&0.0160 	&0.0314 	&0.0487 \\
\hline
\multirow{2}*{GSO}  & IT        &1622	&3324	&5027	&7079	&9858\\
                     & CPU       &0.0022 	&0.0051 	&0.0095 	&0.0156 	&0.0231\\
\hline
\multicolumn{2}{c}{ speed-up$^1$ }      &1.35 	&1.39 	&1.68 	&2.01 	&2.11\\
\hline
\multirow{2}*{RCD}  & IT         &3636	&8113	&14382	&21696	&31904\\
                     & CPU     &0.0195 	&0.0488 	&0.0828 	&0.1320 	&0.1988\\
\hline\multirow{2}*{RGSO}  & IT       &1741	&3580	&5892	&8343	&11745\\
                    & CPU      &0.0114 	&0.0235 	&0.0393 	&0.0598 	&0.0908\\
\hline
\multicolumn{2}{c}{ speed-up$^2$ }     &1.71 	&2.08 	&2.11 	&2.21 	&2.19\\
\bottomrule
\end{tabular}
\end{table}

\begin{table}[!htbp]
\centering
\caption{IT and CPU of CD, RCD, GSO and RGSO for $m \times n$ matrices $A$ with $n = 3000$ and different $m$ when the underdetermined linear system is consistent}
\label{table9}
\renewcommand\arraystretch{0.75}
\begin{tabular}{cccccccc}
\toprule
\multicolumn{2}{c}{ $ m \times n $ }& $100 \times 3000$ & $200 \times 3000 $ & $300 \times 3000$ & $400 \times 3000$ & $500 \times 3000$ \\
\hline
\multirow{2}*{CD}   & IT              &3215	&6924	&11717	&17393	&25296\\
                     & CPU              &0.0029	&0.0069	&0.0138	&0.0238	&0.0384 \\
\hline
\multirow{2}*{GSO}  & IT         &1624	&3267	&4940	&6679	&8683\\
                     & CPU       &0.0025	&0.0051	&0.0099	&0.0146	&0.0214\\
\hline
\multicolumn{2}{c}{ speed-up$^1$ }       &1.19	&1.35	&1.39	&1.63	&1.79\\
\hline
\multirow{2}*{RCD}  & IT         &3475	&7633	&12272	&18248	&25115\\
                     & CPU     &0.0193	&0.0427	&0.0714	&0.1104	&0.1587\\
\hline\multirow{2}*{RGSO}  & IT    &1686	&3499	&5491	&7537	&9966\\
                    & CPU     &0.0107	&0.0225	&0.0377	&0.0532	&0.0724 \\
\hline
\multicolumn{2}{c}{ speed-up$^2$ }    &1.80	&1.90	&1.89	&2.08	&2.19\\
\bottomrule
\end{tabular}
\end{table}

\subsection{Experiments for Random Matrix Collection in $[c,1]$}
\hspace{1.5em}From example \ref{EX3.1}, it can be observed that when the columns of the matrix are nearly linear correlation, the GSO method can find the objective solution of the equation with less iteration steps and running time than the CD method. In order to verify this phenomenon, we construct several $3000\times50$ and $1000\times 3000$ matrices $A$, which entries is independent identically distributed uniform random variables on some interval [c,1]. When the value of $c$ is close to $1$, the column vectors of matrix $A$ are closer to linear correlation. Note that there is nothing special about this interval, and other intervals yield the same results when the interval length remains the same.

From Table \ref{table10} to Table \ref{table12}, it can be seen that no matter whether the system is consistent or inconsistent, overdetermined or underdetermined, with $c$ getting closer to 1, the CD and the RCD method have a significant increase in the number of iterations, and the speed-up$^1$ and the speed-up$^2$ also increase greatly. In Table \ref{table10} and Table \ref{table11}, when $c$ increases to $0.45$, the number of iterations of the CD method exceeds the maximum number of iterations. In Table \ref{table12}, when $c$ increases to $0.6$, the number of iterations of the CD method and RCD method exceeds the maximum number of iterations.

In this group of experiments, it can be observed that when the columns of the matrix are close to linear correlation, the GSO method and the RGSO method can find the least-squares solution more quickly than the CD method and the RCD methd.

\begin{table}[!htbp]
\centering
\caption{IT and CPU of CD, RCD, GSO and RGSO for $A\in R^{3000\times 50}$ with different $c$ when the overdetermined linear system is consistent}
\label{table10}
\renewcommand\arraystretch{0.75}
\begin{tabular}{cccccccc}
\toprule
\multicolumn{2}{c}{ c }&0.15	&0.30&	0.45	&0.60	&0.75&	0.90 \\
\hline
\multirow{2}*{CD}   & IT              &141636	&273589	 &     -	 &     -	   &   -	  &    -\\
                     & CPU           &0.9638 	&1.8351  	 &     -	  &    -	   &   -	  &    - \\
\hline
\multirow{2}*{GSO}  & IT         &12201	&12979	&12763	&11814	&10126	&7017\\
                     & CPU       &0.1575 	&0.1625 	&0.1583 	&0.1519 	&0.1261 	&0.0862\\
\hline
\multicolumn{2}{c}{ speed-up$^1$ }       &6.12 	&11.30	     & -	    &  -	     & -	    &  -\\
\hline
\multirow{2}*{RCD}  & IT        &2196	&3850	&6828	&13978	&36858	&216260\\
                     & CPU       &0.0278 	&0.0483 	&0.0851 	&0.1752 	&0.4506 	&2.6451\\
\hline\multirow{2}*{RGSO}  & IT       &749	&757	&650	&696	&572	&421\\
                    & CPU     &0.0145 	&0.0145 	&0.0124 	&0.0132 	&0.0111 	&0.0079\\
\hline
\multicolumn{2}{c}{ speed-up$^2$ }     &1.92 	&3.33 	&6.87 	&13.22 	&40.68 	&336.90\\
\bottomrule
\end{tabular}
\end{table}

\begin{table}[!htbp]
\centering
\caption{IT and CPU of CD, RCD, GSO and RGSO for $A\in R^{3000\times 50}$ with different $c$ when the overdetermined linear system is inconsistent}
\label{table11}
\renewcommand\arraystretch{0.75}
\begin{tabular}{cccccccc}
\toprule
\multicolumn{2}{c}{ c }&0.15	&0.30&	0.45	&0.60	&0.75&	0.90 \\
\hline
\multirow{2}*{CD}   & IT             &140044	&270445	   &   -	&      -	&      -	  &    -\\
                     & CPU           &0.9483 	&1.8366	&      -	 &     -	&      -	&      - \\
\hline
\multirow{2}*{GSO}  & IT          &12075	&12910	&12678	&11689	&10118	&7112\\
                     & CPU        &0.1602 	&0.1623 	&0.1598 	&0.1519 	&0.1284 	&0.0882\\
\hline
\multicolumn{2}{c}{ speed-up$^1$ }      &5.92 	&11.32  	   &   -	  &    -	&      -	&      -\\
\hline
\multirow{2}*{RCD}  & IT       &2227	&3864	&6493	&14256	&37734	&209427\\
                     & CPU     &0.0301 	&0.0479 	&0.0825 	&0.1783 	&0.4645 	&2.5826\\
\hline\multirow{2}*{RGSO}  & IT     &722	&713	&650	&646	&557	&474\\
                    & CPU     &0.0158 	&0.0145 	&0.0153 	&0.0131 	&0.0121 	&0.0088 \\
\hline
\multicolumn{2}{c}{ speed-up$^2$ }     &1.91 	&3.32 	&5.39 	&13.58 	&38.52 	&292.21\\
\bottomrule
\end{tabular}
\end{table}

\begin{table}[!htbp]
\centering
\caption{IT and CPU of CD, RCD, GSO and RGSO for $A\in R^{1000\times 3000}$ with different $c$ when the underdetermined linear system is consistent}
\label{table12}
\renewcommand\arraystretch{0.75}

\begin{tabular}{cccccccc}
\toprule
\multicolumn{2}{c}{ c }&0.15	&0.30&	0.45	&0.60	&0.75&	0.90 \\
\hline
\multirow{2}*{CD}   & IT             &143373	&246942	&441147	&-	&-	&-\\
                     & CPU           &0.3344 	&0.5818 	&1.0359 	&- 	&- 	&- \\
\hline
\multirow{2}*{GSO}  & IT         &24358	&23888	&22310	&19485	&16795	&11509\\
                     & CPU       &0.1072 	&0.1048 	&0.0987 	&0.0878 	&0.0748 	&0.0525\\
\hline
\multicolumn{2}{c}{ speed-up$^1$ }       &3.12 	&5.55 	&10.49 	&- 	&- 	&-\\
\hline
\multirow{2}*{RCD}  & IT       &122119	&194440	&346301	&-	&-	&-\\
                     & CPU       &0.9120 	&1.4251 	&2.5529 	&- 	&-	&-\\
\hline\multirow{2}*{RGSO}  & IT       &28166	&24936	&24201	&22318	&18433	&13717\\
                    & CPU    &0.2711 	&0.2450 	&0.2318 	&0.2082 	&0.1813 	&0.1306\\
\hline
\multicolumn{2}{c}{ speed-up$^2$ }    &3.36 	&5.82 	&11.01 	&-	&- 	&-\\
\bottomrule
\end{tabular}
\end{table}

\section{Conclusion}
\hspace{1.5em}A new extension of the CD method and its randomized version, called the GSO method and the RGSO method, are proposed for solving the linear least-squares problem. The GSO method is deduced to be convergent, and an estimate of the convergence rate of the RGSO method is obtained. The GSO method and the RGSO method are proved to converge faster than the CD method and the RCD method, respectively. Numerical experiments show the effectiveness of the two methods, especially when the columns of coefficient matrix $A$ are close to linear correlation.
\section*{Acknowledgments}
\hspace{1.5em}This work was supported by the National Key Research and Development Program of China [grant number 2019YFC1408400], and the Science and Technology Support Plan for Youth Innovation of University in Shandong Province [No.YCX2021151].

\end{document}